\documentclass[12pt]{amsart}
\usepackage{amssymb}
\usepackage{amscd}
\usepackage{amsmath}
\usepackage{enumerate}
\usepackage[all]{xy}
\usepackage{verbatim}

\newtheorem{theo}{Theorem}[section]
\newtheorem{lemma}[theo]{Lemma}
\newtheorem{cor}[theo]{Corollary}
\newtheorem{prop}[theo]{Proposition}

\newtheorem{remark}[theo]{Remark}

\def\e{{\epsilon}}

\newcommand{\B}{{\mathbb{B}}}
\newcommand{\C}{{\mathbb{C}}}
\newcommand{\R}{{\mathbb{R}}}

\newcommand{\BS}{\mathbb{S}}
\newcommand{\D}{\mathbb{D}}
\newcommand{\Z}{\mathbb{Z}}

\begin{document}
\title{ON THE L\^E-MILNOR FIBRATION FOR REAL ANALYTIC MAPS}
\author{Aur\'elio Menegon Neto}
\author{Jos\'e Seade}

\thanks{This work was realized with support from CNPq, ``Conselho Nacional de Desenvolvimento Cient\'ifico e Tecnol\'ogico" - Brazil}

\address{Aur\'elio Menegon Neto: Departamento de Matem\'atica - Universidade Federal da Para\'iba - Brasil.}
\email{aurelio@mat.ufpb.br}

\address{Jos\'e Seade: Instituto de Matem\'aticas - Universidad Nacional Aut\'onoma de M\'exico.}
\email{jseade@im.unam.mx}

\begin{abstract}
In this paper, we study the topology of real analytic map-germs with isolated critical value $f: (\R^m,0) \to (\R^n,0)$, with $1 <n <m$. We compare the topology of $f$ with the topology of the compositions $\pi_i^* \circ f$, where $\pi_i^*: \R^n \to \R^{n-1}$ are the projections $(t_1, \dots, t_n) \mapsto (t_1, \dots, t_{i-1}, t_{i+1}, \dots, t_n)$, for $i=1, \dots, n$. As a main result, we give necessary and sufficient conditions for $f$ to have a L\^e-Milnor fibration in the tube.
\end{abstract}
\maketitle

\section{Introduction}  \label{section_1}

In \cite{Mi} J. Milnor pointed out that if $f: (\R^m,0) \to (\R^n,0)$, with $1< n <m$, is a real analytic map-germ with an isolated critical point and if $\pi_i^*: \R^n \to \R^{n-1}$ is the projection obtained by removing the $i$-th coordinate, then the composition: 
$$f_i^* := \pi_i^* \circ f: (\R^m,0) \to (\R^{n-1},0)$$
also has an isolated critical point. He conjectured that the Milnor fiber associated to this composition is homeomorphic to the product of the Milnor fiber of $f$ and the closed unit interval.

In his Ph.D. thesis \cite{Ki}, H.C. King proved Milnor's conjecture. In fact, he proved that the Milnor fiber of $\pi_i^* \circ f$ is diffeomorphic to the product of the Milnor fiber of $f$ and the closed unit interval (with corners rounded off).

Later, N. Dutertre and R.N.A. dos Santos extended this result for real analytic map-germs $f$ with an isolated critical value and with Milnor's condition $(b)$. As defined by Massey in \cite{Ma}, we say that $f$ satisfies Milnor's condition $(b)$ if
$0$ is isolated in $V \cap \overline{\Sigma_{f,\rho} \backslash V}$, where $\rho$ is the square of the distance function to the origin and $\Sigma_{f,\rho}$ is the set of critical points of the pair $(f, \rho)$. Massey showed that any real analytic map-germ $f$ with an isolated critical value and with Milnor's condition $(b)$ has a L\^e-Milnor fibration:
$$f_|: f^{-1}(\B_\eta^n \backslash \{0\}) \cap \B_\e^m \to (\B_\eta^n \backslash \{0\}) \cap \Im(f) \, .$$
One can easily check that if $f$ has an isolated critical point at $0 \in \R^m$, then $f$ satisfies Milnor's condition $(b)$.

N. Dutertre and R.N.A. dos Santos proved in \cite{DS} that if $f: (\R^m,0) \to (\R^n,0)$ has an isolated critical value and satisfies Milnor's condition $(b)$, then $f_i^*: (\R^m,0) \to (\R^{n-1},0)$ also has an isolated critical value and Milnor's condition $(b)$, and the Milnor fiber of $f_i^*$ is homeomorphic to the product of the Milnor fiber of $f$ and the closed unit interval.

The idea of composing a map-germ $f$ with a projection to study the topology of $f$ was also used by J.L. Cisneros-Molina, J. Seade and J. Snoussi in several papers (\cite{CSS1}, \cite{CSS2} and \cite{CSS3}). 

In this paper, we use the idea of composing a real analytic map-germ $f$ with the projections $\pi_i^*$ to study the existence of a L\^e-Milnor fibration in a tube.

For any positive real number $r>0$ let $\B_r^n$ denote the closed ball around $0$ in $\R^n$ of radius $r$. Our main theorem is:

\begin{theo} \label{theo_main}
Let $f: (\R^m,0) \to (\R^n,0)$, with $1<n<m$, be a real analytic map-germ with an isolated critical value at $0 \in \R^n$. If $\e>0$ is small enough, then the following conditions are equivalent:
\medskip
\begin{itemize}
\item[$(i)$] The restriction:
$$f_|: f^{-1}(\B_\eta^n \backslash \{0\}) \cap \B_\e^m \to \B_\eta^n \backslash \{0\} \cap \Im(f)$$
is the projection of a smooth locally trivial fibration, where $\Im(f)$ is the image of $f$.
\medskip
\item[$(ii)$] For any $i=1, \dots, n$ and for any $t \in (\B_\eta^n \backslash \{0\}) \cap \Im(f)$ with $t_i^* \neq 0$, the restrictions:
\medskip
\begin{itemize}
\item[$-$] $(f_i^*)_|: (f_i^*)^{-1}(\B_\eta^{n-1} \backslash \{0\}) \cap \B_\e^m \to (\B_\eta^{n-1} \backslash \{0\}) \cap \Im(f_i^*)$
\item[$-$] $(f_i)_|: (f_i)^{-1}(\B_\eta^1 \backslash \{0\}) \cap \B_\e^m \to (\B_\eta^1 \backslash \{0\}) \cap \Im(f_i)$
\item[$-$] $(f_i)_| : (f_i)^{-1}(\B_\eta^1) \cap (f_i^*)^{-1}(t_i^*) \cap \B_\e^m \to \B_\eta^1 \cap \Im(f_{i |})$
\end{itemize}
\medskip
are projections of smooth locally trivial fibrations, where $t_i^* := \pi_i^*(t)$.
\medskip
\item[$(iii)$] The diffeomorphism type of $(f_i)^{-1}(t_i) \cap \B_\e^m$ is independent of $i=1, \dots, n$ and of $t_i \in (\B_\eta^1 \backslash \{0\}) \cap \Im(f_i)$, and it is the same diffeomorphism type of $(f^{-1}(t) \cap \B_\e^m) \times \B_\eta^{n-1}$, for any $t \in (\B_\eta^n \backslash \{0\}) \cap \Im(f)$ with $t_i \neq 0$.
\medskip
\item[$(iv)$] The Euler characteristic of $f^{-1}(t) \cap \BS_\e^{m-1}$ is constant, for any $t \in (\B_\eta^n \backslash \{0\}) \cap \Im(f)$.
\end{itemize}
\end{theo}

\medskip
So Theorem \ref{theo_main} provides a sufficient (and obviously necessary) condition for a map-germ $f$ with isolated critical value to have a L\^e-Milnor fibration in the tube: that the Euler characteristic of the boundary of the ``fibers" over a neighborhood of $0 \in \R^n$ is constant.

Moreover, as a corollary, it gives a stronger version of Dutertre-Santos' result:

\begin{cor} \label{cor_2}
Let $f: (\R^m,0) \to (\R^n,0)$, with $1<n<m$, be a real analytic map-germ with an isolated critical value and with a L\^e-Milnor fibration in the tube. Then for any set of indices $\{ i_1, \dots, i_k \}$ we have that $F_{i_1, \dots, i_k}$ is diffeomorphic to $F \times \B_\eta^{n-k}$ (with corners rounded off), where $F$ is the Milnor fiber of $f$ and $F_{i_1, \dots, i_k}$ is the Milnor fiber of the map-germ $(f_{i_1}, \dots, f_{i_k})$.
\end{cor}

The equivalence between $(i)$ and $(iv)$ of Theorem \ref{theo_main} has a similar nature of a theorem of \cite{Su} and \cite{HL}, which says that if $f: \C^2 \to \C$ is a polynomial function, then a regular value $a \in \C$ of $f$ is not in the bifurcation set of $f$ if and only if the Euler characteristic of the fibers $\chi(f^{-1}(t))$ is constant for $t$ varying in some neighborhood of $a$. We mention that C. Joita and M. Tibar gave a real counterpart for that theorem in \cite{JT}.

\medskip
This article is organized in the following way:

In section 2, we consider the more general situation of a real analytic map (instead of a map-germ). We study the general case when the discriminant of $f$ is arbitrary.

In section 3, we consider a map-germ $f: (\R^m,0) \to (\R^n,0)$ with an isolated critical value. We prove Theorem \ref{theo_main}.

In section 4, we study the topology of a real analytic map-germ $f: (\R^m,0) \to (\R^n,0)$ with a L\^e-Milnor fibration, using the Euler characteristic with compact support. We relate the Euler characteristic with compact support of the link $K$ of $f$ with the Euler characteristic of the Milnor fiber. Precisely, we prove:

\begin{theo} \label{theo_link}
Let $f: (\R^m,0) \to (\R^n,0)$, with $1 \leq n < m$, be a real analytic map-germ with an isolated critical value and with a L\^e-Milnor fibration. If $1<n$ we have:
$$
\chi_C(K) =
\begin{cases}
\hspace{0.8cm} 0 \hspace{1cm} , \ {\text if} \ m \in 2 \Z \ \ {\text and} \ n \in 2 \Z \ ; \\
\hspace{.42cm} 2 \chi(F) \hspace{.42cm} , \ {\text if} \ m \in 2 \Z \ \ {\text and} \ n \in 2 \Z+1 \ ; \\
\hspace{0.8cm} 2 \hspace{1cm} , \ {\text if} \ m \in 2 \Z+1 \ \ {\text and} \ n \in 2 \Z \ ; \\
2 -2 \chi(F) \ , \ {\text if} \ m \in 2 \Z+1 \ \ {\text and} \ n \in 2 \Z+1 \ . \\
\end{cases} 
$$

If $n=1$ we have:
$$
\chi_C(K) =
\begin{cases}
\hspace{.42cm} \chi(F_+) + \chi(F_-) \hspace{.42cm} , \ {\text if} \ m \in 2 \Z \ ; \\
2 -\chi(F_+) - \chi(F_-) \ , \ {\text if} \ m \in 2 \Z+1  \ , \\
\end{cases} 
$$
where $F_+$ is the Milnor fiber to the right, that is, $F_+ := f^{-1}(t) \cap \B_\e$ for $t>0$ sufficiently small, and $F_-$ is the Milnor fiber to the left, that is, $F_- := f^{-1}(t) \cap \B_\e$ for $t<0$ sufficiently small.
\end{theo}

A quite similar result was already proved by N. Dutertre and R.N.A. dos Santos in \cite{DS}, when $n \geq 2$. Also, in \cite{Ha} H.A. Hamm gave similar results when $f$ has an isolated critical point.

Later, we restrict to the case when $f$ has an isolated critical point, and we obtain:

\begin{cor} \label{cor_EC}
Let $f: (\R^m,0) \to (\R^n,0)$, with $1 \leq n<m$, be a real analytic map-germ with an isolated critical point at $0 \in \R^m$. 
\begin{itemize}
\item[$(i)$] If $n>1$ and if $m$ is odd, then $\chi(F)=1$. In particular, the Milnor fiber $F$ does not have the homotopy type of a bouquet of spheres of the same dimension.
\item[$(ii)$] If $n=1$ and if $m$ is odd, then $\chi(F_+) + \chi(F_-) =2$. In particular, the ``Milnor fibers" $F_+$ and $F_-$ of $f$ cannot have both the homotopy type of bouquets of spheres of the same dimension.
\item[$(iii)$] If $n=1$ and if $m$ is even, then $\chi(F_+) = \chi(F_-)$.
\end{itemize}
\end{cor}

The item $(i)$ of Corollary \ref{cor_EC} was already proved by dos Santos, Dreibelbis and Dutertre in \cite{SDD}, using Morse theory for manifolds with boundary, and item $(iii)$ was proved by Hamm in \cite{Ha}.

In section 5, we use the results of section 3 to answer a question made by J.L. Cisneros-Molina, J. Seade and J. Snoussi in \cite{SCG}. Precisely, we prove that the diffeo\-morphism type of the link of the singular germ $(V_i^*,0)$ given by the map-germ $f_i^*: (\R^m,0) \to (\R^{n-1},0)$ obtained by removing from $f$ the coordinate function $f_i$, is independent of $i \in \{1, \dots, n\}$. This is contained in Theorem \ref{theo_SCG}.

\medskip
\section{The general situation}

Let $f: \R^m \to \R^n$, with $1\leq n <m$, be a real analytic map, and let $\Delta$ be the discriminant of $f$, that is, the set of critical values of $f$. We say that $f$ has a L\^e-Milnor fibration if there exists a real number $\e_0>0$ such that, for any $\e \leq \e_0$, there exists $\eta$, with $0<\eta \ll \e$, such that the restriction:
$$f_|: f^{-1}(\B_\eta^n \backslash \Delta) \cap \B_\e^m \to (\B_\eta^n \backslash \Delta) \cap \Im(f)$$
is the projection of a smooth locally trivial fibration, where $\Im(f)$ denotes the image of the map $f_|$.

In this section, we study the existence of a L\^e-Milnor fibration for real analytic maps with arbitrary discriminant set. Later, in the next sections, we restrict to the local situation, that is, we consider map-germs $f: (\R^m,0) \to (\R^n,0)$, and we suppose that $f$ has isolated critical value. Although this path leads to more complicate notations and some extra considerations, we chose to start with this more general situation. 

Naturally, we should start with the simplest case of a function $f: \R^m \to \R$. The existence of a L\^e-Milnor fibration in this case is an immediate consequence of the combination of Sard's theorem with Ehresmann's fibration lemma. We write it bellow, for completeness reasons:

\begin{lemma} \label{lemma_0}
Let $f: \R^m \to \R$ be a real analytic function and suppose that $f(0)=0$. Then for any positive real number $\e>0$ there exists a positive real number $\eta>0$ such that the restriction:
$$f_|: f^{-1}(\B_\eta^1 \backslash \{0\}) \cap \B_\e^m \to \B_\eta^1 \backslash \{0\} \cap \Im(f)$$
is the projection of a locally trivial fibration, where $\Im(f)$ denotes the image of $f$.
\end{lemma}

\begin{proof}
 It is well-known that if $g: M \to \R$ is a real analytic function defined on a limited smooth manifold $M \subset \B_\delta$, for some closed ball $\B_\delta$ in $\R^m$, then $g$ has at most a finite number of critical values in $\R$ (see \cite{Mi} and \cite{Ve} for instance).

Now let $\tilde{f}$ be the restriction of $f$ to the open ball $\mathring{\B}_{2\e}^m$ and let $\tilde{\tilde{f}}$ be the restriction of $f$ to the sphere $\BS_\e^{m-1}$. By the previous paragraph, both $\tilde{f}$ and $\tilde{\tilde{f}}$ have at most an isolated critical value at $0 \in \R$. So there exists $\eta>0$ sufficiently small such that $0$ is the only critical value of both $\tilde{f}$ and $\tilde{\tilde{f}}$ in $\B_\eta^1$. This means that $f^{-1}(t) \cap \mathring{\B}_{2\e}^m$ is a smooth manifold that intersects $\BS_\e^{m-1}$ transversally, for any $t \in \B_\eta^1 \backslash \{0\} \cap \Im(f)$. Then the result follows from Ehresmann's fibration lemma.

\end{proof}

Now let us consider the case $n>1$. First, we shall establish some useful notations. Given an Euclidian space $\R^n$, with $n>1$, consider the canonic projections:
$$
\begin{array}{cccc}
\pi_i \ : & \! \R^n & \! \longrightarrow & \! \R \\
& \! (t_1, \dots, t_n) & \! \longmapsto & \! t_i
\end{array}
$$
and:
$$
\begin{array}{cccc}
\pi_i^* \ : & \! \R^n & \! \longrightarrow & \! \R^{n-1} \\
& \! (t_1, \dots, t_n) & \! \longmapsto & \! (t_1, \dots, t_{i-1}, t_{i+1}, \dots, t_n)
\end{array}
\, ,$$
for each $i=1, \dots, n$. 

In order to simplify notation, we set: 
$$t_i^* := \pi_i^*(t) \, .$$

\medskip
Now, given a map $f = (f_1, \dots, f_n): \R^m \to \R^n$ set:
$$f_i^* := (\pi_i^* \circ f): \R^m \to \R^{n-1} $$
for each $i= 1, \dots, n$. That is, $f_i^* = (f_1, \dots, f_{i-1}, f_{i+1}, \dots, f_n)$. 

We have:

\begin{remark} \label{remark_1}
For any $i \in \{1, \dots, n\}$ the critical set $\Sigma(f)$ of $f$ is the union of:
\begin{itemize}
\item[$\circ$] the critical set $\Sigma(f_i)$ of $f_i$ ;
\item[$\circ$] the critical set $\Sigma(f_i^*)$ of $f_i^*$ ;
\item[$\circ$] the set of points $x \in \R^m$ such that $x$ is a regular point of both $f_i$ and $f_i^*$ and such that $(f_i)^{-1} \big( f_i(x) \big)$ intersects $(f_i^*)^{-1}\big( (f_i^*)(x) \big)$ not transversally at $x$ in $\R^m$.
\end{itemize}
\end{remark}

\medskip
The main idea of this paper is to read the existence of a L\^e-Milnor fibration for $f: \R^m \to \R^n$ in the topology of the maps $f_i^*: \R^m \to \R^{n-1}$. So the first step is to relate the transversality of the ``fibers" of $f$ with the spheres centered at the origin and the transversality of the ``fibers" of $f_i^*$ and of $f_i$, as well as their intersections, with that spheres. 

In this sense, we have:

\begin{lemma} \label{lemma_01}
Let $A$, $B$ and $N$ be sub-manifolds of a manifold $M$, and set $A':= A \cap N$ and $B':= B \cap N$. Then for any point $x \in A \cap B \cap N$, the following conditions are equivalent:
\begin{itemize}
\item[$(i)$] $B$ intersects $N$ transversally in $M$ at $x$ and $A'$ intersects $B'$ transversally in $N$ at $x$;
\item[$(ii)$] $A'$ intersects $B$ transversally in $M$ at $x$.
\end{itemize}
\end{lemma}

\begin{proof}
Suppose that $B$ intersects $N$ transversally in $M$ at $x$ and $A'$ intersects $B'$ transversally in $N$ at $x$. Then $T_x(M) = T_x(B) + T_x(N)$ and $T_x(N) = T_x(A') + T_x(B')$. So if $\vec u$ is a vector of $T_x(M)$, there exist vectors $\vec v \in T_x(N)$ and $\vec w \in T_x(B)$ such that $\vec u = \vec v + \vec w$. But $\vec v$ can be written as a sum $\vec k + \vec l$, for some $\vec k \in T_x(A')$ and $\vec l \in T_x(B')$. So $\vec u = \vec k + (\vec l + \vec w)$, where the vector $\vec l + \vec w$ is in $T_x(B)$.

Conversely, suppose that $A'$ intersects $B$ transversally in $M$ at $x$. Then $T_x(M) = T_x(A') + T_x(B) \subset T_x(N) + T_x(B)$, and hence $B$ intersects $N$ transversally in $M$ at $x$. Moreover, given a vector $\vec u \in T_x(N)$, there exist vectors $\vec v \in T_x(A') \subset T_x(N)$ and $\vec w \in T_x(B)$ such that $\vec u = \vec v + \vec w$. But clearly $\vec w$ must be also in $T_x(N)$ and hence it is in $T_x(B') = T_x(B) \cap T_x(N)$. 

\end{proof}

Bringing the lemma above to our context, we get:

\begin{lemma} \label{lemma_1}
Let $f: \R^m \to \R^n$ be a real analytic map, with $1<n<m$, and let $x \neq 0$ be a regular point of $f$ such that:
\begin{itemize}
\item[$(a)$] $T_x \big( (f_i)^{-1}(t_i) \big)$ intersects $T_x (\BS_{||x||})$ transversally in $T_x(\R^m)$,
\end{itemize}
for some $i=1, \dots, n$, where $t = f(x)$ and $\BS_{||x||}$ is the sphere around $0$ in $\R^m$ that contains $x$. Then the following conditions are equivalent:
\begin{itemize}
\item[$(c)$] $[ T_x \big( (f_i^*)^{-1}(t_i^*) \big) \cap T_x (\BS_{||x||}) ]$ intersects $[ T_x \big( (f_i)^{-1}(t_i) \big) \cap T_x (\BS_{||x||}) ]$ transversally in $T_x(\BS_{||x||})$.
\item[$(d)$] $[ T_x \big( (f_i^*)^{-1}(t_i^*) \big) \cap T_x (\BS_{||x||}) ]$ intersects $T_x \big( (f_i)^{-1}(t_i) \big)$ transversally in $T_x(\R^m)$.
\end{itemize}
\end{lemma}

Now we can prove:

\begin{prop} \label{prop_3}
Let $f: \R^m \to \R^n$ be a real analytic map, with $1<n<m$ and with $f(0)=0$. Let $\e$ and $\eta$ be positive real numbers and let $\Delta$ be the set of critical values of $f$ in $\B_\eta^n$. The following conditions are equivalent:
\medskip
\begin{itemize}
\item[$(i)$] The restriction:
$$f_|: f^{-1}(\B_\eta^n \backslash \Delta) \cap \B_\e^m \to (\B_\eta^n \backslash \Delta) \cap \Im(f)$$
is the projection of a smooth locally trivial fibration.
\medskip
\item[$(ii)$] For each $t \in (\B_\eta^n \backslash \Delta) \cap \Im(f)$ and for each $x \in \big( f^{-1}(t) \cap \BS_\e^{m-1} \big)$, and for any $i = 1, \dots, n$ we have that:
\begin{itemize}
\item[$(a)$] $T_x \big( (f_i)^{-1}(t_i) \big)$ intersects $T_x (\BS_\e^{m-1})$ transversally at $x$;
\item[$(b)$] $T_x \big( (f_i^*)^{-1}(t_i^*) \big)$ intersects $T_x (\BS_\e^{m-1})$ transversally at $x$;
\item[$(e)$] For any $s \in \pi_i^*(\B_\eta^n)$ the restriction:
$$(f_i)_| : (f_i)^{-1} \big( \B_\eta^1 \backslash \pi_i(\Delta \cap H_s) \big) \cap (f_i^*)^{-1}(s) \cap \B_\e^m \to \big( \B_\eta^1 \backslash \pi_i(\Delta \cap H_s) \big) \cap \Im(f_{i |})$$
is the projection of a smooth locally trivial fibration, where $H_s:= (\pi_i^*)^{-1}(s)$ and $\Im(f_{i |})$ is the image of such restriction.
\end{itemize}
\end{itemize}
\end{prop}

\begin{proof}
By Remark \ref{remark_1}, for any $x \in f^{-1}(\B_\eta^n \backslash \Delta)$ and for any $i=1, \dots, n$ we have that $x$ is a regular point of both $f_i$ and $f_i^*$ and that $T_x \big( (f_i)^{-1}(f_i(x)) \big)$ intersects $T_x \big( (f_i^*)^{-1}(f_i^*(x)) \big)$ transversally in $T_x(\R^m)$.

\begin{itemize}
\medskip
\item[$\bullet$] Clearly $(i)$ implies that $f^{-1}(t)$ intersects $\BS_\e^{m-1}$ transversally, for any $t \in (\B_\eta^n \backslash \Delta) \cap \Im(f)$. Since for any $x \in f^{-1}(t) \cap \BS_\e^{m-1}$, with $t \in (\B_\eta^n \backslash \Delta) \cap \Im(f)$, and for any $i=1, \dots, n$ fixed we have that: 
$$T_x \big( (f_i)^{-1}(t_i) \big) \cap T_x \big( (f_i^*)^{-1}(t_i^*) \big) = T_x \big( f^{-1}(t) \big) \, ,$$
it follows that both $T_x \big( (f_i)^{-1}(t_i) \big)$ and $T_x \big( (f_i^*)^{-1}(t_i^*) \big)$ intersect $\BS_\e^{m-1}$ transversally at $x$. These are the conditions $(a)$ and $(b)$ above. 

\medskip
\noindent Moreover, $(i)$ also implies that the restriction of $f$ to the sphere $\BS_\e^{m-1}$ is a submersion. By Remark \ref{remark_1}, this implies that the restriction of $f_i$ to $\BS_\e^{m-1}$ and the restriction of $f_i^*$ to $\BS_\e^{m-1}$ are submersions, and that:
$$[ T_x \big( (f_i^*)^{-1}(t_i^*) \big) \cap T_x (\BS_\e^{m-1}) ] \ \textrm{intersects} \ [ T_x \big( (f_i)^{-1}(t_i) \big) \cap T_x (\BS_\e^{m-1}) ] \ \textrm{transversally  in} \ T_x(\BS_\e^{m-1}) \, , $$
which is the condition $(c)$ of Lemma \ref{lemma_1}. So it follows from that lemma that: 
$$[ T_x \big( (f_i^*)^{-1}(t_i^*) \big) \cap T_x (\BS_\e^{m-1}) ] \ \textrm{intersects} \ T_x \big( (f_i)^{-1}(t_i) \big) \ \textrm{transversally  in} \ T_x(\R^m) \, , \leqno(*)$$
for any $x \in f^{-1}(t) \cap \BS_\e^{m-1}$ with $t \in (\B_\eta^n \backslash \Delta) \cap \Im(f)$.

\medskip
\noindent Now fix some $s \in \pi_i^*(\B_\eta^n)$ and set:
$$M:= (f_i)^{-1} \big( \B_\eta^1 \backslash \pi_i(\Delta \cap H_s) \big) \cap (f_i^*)^{-1}(s) \, .$$
We are going to show that $M$ is a smooth manifold. In fact, we know that both $(f_i^*)^{-1}(s) \setminus \Sigma(f_i^*)$ and $(f_i)^{-1}(\B_\eta^1) \setminus \Sigma(f_i)$ are smooth manifolds. So we just have to show that:

\medskip
\begin{itemize}
\item[$(I)$] $\Sigma(f_i^*) \cap (f_i^*)^{-1}(s)$ is contained in $(f_i)^{-1}\big(\pi_i(\Delta \cap H_s)\big)$; and:
\item[$(II)$] $\Sigma(f_i) \cap (f_i^*)^{-1}(s)$ is contained in $(f_i)^{-1}\big(\pi_i(\Delta \cap H_s)\big)$.
\end{itemize}

\medskip
\noindent To show $(I)$, take $x \in \Sigma(f_i^*) \cap (f_i^*)^{-1}(s)$. Then $x \in \Sigma(f)$ and hence $f(x) \in \Delta$. But we also have that $f(x) \in H_s$ and hence $f(x) \in (\Delta \cap H_s)$. Therefore $f_i(x) \in \pi_i(\Delta \cap H_s)$.

\medskip
\noindent To show $(II)$, take $x \in \Sigma(f_i) \cap (f_i^*)^{-1}(s)$. Then $x \in \Sigma(f)$ and hence $f(x) \in (\Delta \cap H_s)$. Therefore $f_i(x) \in \pi_i(\Delta \cap H_s)$.

\medskip
\noindent Now we want to apply the Ehresmann's fibration lemma to show that the restriction:
$$(f_i)_| : (f_i)^{-1} \big( \B_\eta^1 \backslash \pi_i(\Delta \cap H_s) \big) \cap (f_i^*)^{-1}(s) \cap \B_\e^m \to \big( \B_\eta^1 \backslash \pi_i(\Delta \cap H_s) \big) \cap \Im(f_{i |})$$
is the projection of a locally trivial fibration, where $\Im(f_{i |}) = f_i(M \cap \B_\e)$ is the image of such restriction. So besides having the property $(II)$ above, we also have to show that:

\medskip
\begin{itemize}
\item[$(III)$] for any $t_i \in \big( \B_\eta^1 \backslash \pi_i(\Delta \cap H_s) \big) \cap \Im(f_{i |})$ and for any $x \in (f_i)^{-1}(t_i) \cap (f_i^*)^{-1}(s) \cap \B_\e^m$, one has that $(f_i)^{-1}(t_i)$ intersects $(f_i^*)^{-1}(s)$ transversally at $x$;
\item[$(IV)$] for any $t_i \in \big( \B_\eta^1 \backslash \pi_i(\Delta \cap H_s) \big) \cap \Im(f_{i |})$ and for any $x \in (f_i)^{-1}(t_i) \cap (f_i^*)^{-1}(s) \cap \BS_\e^{m-1}$, one has that $(f_i)^{-1}(t_i) \cap (f_i^*)^{-1}(s)$ intersects $\BS_\e^{m-1}$ transversally at $x$.
\end{itemize}

\medskip
\noindent To show $(III)$, let $t_i$ be a point in $\B_\eta^1 \cap \Im(f_{i |})$ and suppose that there exists $x \in (f_i)^{-1}(t_i) \cap (f_i^*)^{-1}(s) \cap \B_\e^m$ such that $(f_i)^{-1}(t_i)$ intersects $(f_i^*)^{-1}(s)$ not transversally at $x$. This means that $(f_i)^{-1}(f_i(x))$ intersects $(f_i^*)^{-1}\big( (f_i^*)(x) \big)$ not transversally at $x$. So $x \in \Sigma(f)$, which implies that $f(x) \in (\Delta \cap H_s)$. So $t_i = f_i(x)$ belongs to $\pi_i(\Delta \cap H_s)$.

\medskip
\noindent To show $(IV)$, let $t_i$ be a point in $\B_\eta^1 \cap \Im(f_{i |})$ and suppose that there exists $x \in (f_i)^{-1}(t_i) \cap (f_i^*)^{-1}(s) \cap \BS_\e^{m-1}$ such that $(f_i)^{-1}(t_i) \cap (f_i^*)^{-1}(s)$ intersects $\BS_\e^{m-1}$ not transversally at $x$. Then $(f_i)^{-1}(f_i(x)) \cap (f_i^*)^{-1}\big( (f_i^*)(x) \big)$ intersects $\BS_\e^{m-1}$ not transversally at $x$. So it follows from $(*)$ that $f(x) \in \Delta$. Hence $f(x) \in (\Delta \cap H_s)$ and so $t_i \in \pi_i(\Delta \cap H_s)$.

\medskip
\item[$\bullet$] Now let us show that $(ii)$ implies $(i)$. Notice that for any $t \in \D_\eta^n$ it is true that:
$$t_i \in \pi_i(\Delta \cap H_{t_i^*}) \Leftrightarrow t \in (\Delta \cap H_{t_i^*}) \Leftrightarrow t \in \Delta \, .$$
Now fix an arbitrary $t \in (\B_\eta^n \backslash \Delta) \cap \Im(f)$. Since $t_i \notin \pi_i(\Delta \cap H_{t_i^*})$ it follows from the hypothesis that $(f_i)^{-1}(t_i)$ intersects $(f_i^*)^{-1}(t_i^*) \cap \BS_\e^{m-1}$ transversally. Hence the conditions $(a)$ and $(d)$ of Lemma \ref{lemma_1} are satisfied. So it follows from that lemma that the condition $(c)$ is also satisfied. This gives that the restriction of $f$ to $\BS_\e^{m-1}$ is a submersion (see Remark \ref{remark_1}). So the result follows from Ehresmann's fibration lemma.
\end{itemize}

\end{proof}

\begin{remark}
Actually, we proved that if there is some $i \in \{1, \dots, n\}$ such that for each $t \in (\B_\eta^n \backslash \Delta) \cap \Im(f)$ and for each $x \in \big( f^{-1}(t) \cap \BS_\e^{m-1} \big)$ we have that the conditions $(a)$, $(b)$ and $(e)$ are satisfied, then the conditions $(i)$ is also satisfied, that is, $f$ has a L\^e-Milnor fibration.
\end{remark}

\medskip
\section{Map-germs with isolated critical value}

Now let $f: (\R^m,0) \to (\R^n,0)$, with $1<n<m$, be a real analytic map-germ with an isolated critical value. This means that there exist positive real numbers $\e$ and $\eta_0$ such that the restriction: 
$$f_|: f^{-1}(\mathring{\B}_{\eta_0}^n \backslash \{0\}) \cap \mathring{\B}_\e^m \to \mathring{\B}_{\eta_0}^n \backslash \{0\}$$
is a submersion.

This implies that there is no critical point of $f_i^*$ in $f^{-1}(\mathring{\B}_{\eta_0}^n \backslash \{0\}) \cap \mathring{\B}_\e^m$, for any $i=1, \dots, n$. But this does not mean that there is no critical point of $f_i^*$ in $(f_i^*)^{-1}(\mathring{\B}_{\eta_0}^{n-1} \backslash \{0\}) \cap \mathring{\B}_\e^m$, since $f^{-1}(\mathring{\B}_{\eta_0}^n) \cap \mathring{\B}_\e^m$ is properly contained in $(f_i^*)^{-1}(\mathring{\B}_{\eta_0}^{n-1}) \cap \mathring{\B}_\e^m$.

Nevertheless, if we take $\e_0>0$ sufficiently small such that: 
$$\B_{\e_0}^m \subset f^{-1}(\mathring{\B}_{\eta_0}^n) \cap \B_\e^m \, ,$$
we get that there is no critical point of $f_i^*$ in $(f_i^*)^{-1}(\mathring{\B}_{\eta_0}^{n-1} \backslash \{0\}) \cap \B_{\e_0}^m$. Also, there is no critical point of $f_i$ in $(f_i)^{-1}(\B_{\eta_0}^1 \backslash \{0\}) \cap \mathring{\B}_{\e_0}^m$.

\medskip
So we have:

\begin{lemma} \label{lemma_icv}
Let $f: (\R^m,0) \to (\R^n,0)$, with $1<n<m$, be a real analytic map-germ with an isolated critical value. Then there exist positive real numbers $\e_0$ and $\eta_0$ such that, for any $\e$ and for any $\eta$ with $0< \e \leq \e_0$ and $0< \eta \leq \eta_0$, the restrictions: 
$$f_|: f^{-1}(\mathring{\B}_\eta^n \backslash \{0\}) \cap \mathring{\B}_\e^m \to \mathring{\B}_\eta^n \backslash \{0\} \, ,$$
$$(f_i^*)_|: (f_i^*)^{-1}(\mathring{\B}_\eta^{n-1} \backslash \{0\}) \cap \mathring{\B}_\e^m \to \mathring{\B}_\eta^{n-1} \backslash \{0\} \, ,$$
and
$$(f_i)_|: (f_i)^{-1}(\mathring{\B}_\eta^1 \backslash \{0\}) \cap \mathring{\B}_\e^m \to \mathring{\B}_\eta^1 \backslash \{0\} \, ,$$
are submersions, for every $i=1, \dots, n$.
\end{lemma}

\medskip
So now we can prove Theorem \ref{theo_main}:

\medskip
\noindent {\it Proof of Theorem \ref{theo_main}.} Let $\e>0$ be small enough as in Lemma \ref{lemma_icv}. We will show that $(i)$ is equivalent to $(ii)$ and that $(i) \implies (iii) \implies (iv) \implies (ii)$.

\medskip
\begin{itemize}
\item[$\bullet$] Let us show that $(i) \implies (ii)$. By Proposition \ref{prop_3}, for each $t \in (\B_\eta^n \backslash \{0\}) \cap \Im(f)$ and for each $x \in \big( f^{-1}(t) \cap \BS_\e^{m-1} \big)$, and for any $i = 1, \dots, n$ we have that:
\medskip
\begin{itemize}
\item[$(a)$] $T_x \big( (f_i)^{-1}(t_i) \big)$ intersects $T_x (\BS_\e^{m-1})$ transversally at $x$;
\item[$(b)$] $T_x \big( (f_i^*)^{-1}(t_i^*) \big)$ intersects $T_x (\BS_\e^{m-1})$ transversally at $x$;
\item[$(e)$] If $t_i^* \neq 0$ then the restriction:
$$(f_i)_| : (f_i)^{-1}(\B_\eta^1) \cap (f_i^*)^{-1}(t_i^*) \cap \B_\e^m \to \B_\eta^1 \cap \Im(f_{i |})$$
is the projection of a smooth locally trivial fibration.
\end{itemize}

\medskip
\noindent But conditions $(a)$ and $(b)$ together with Lemma \ref{lemma_icv} and with Ehresmann's fibration lemma give that the restrictions:
$$(f_i^*)_|: (f_i^*)^{-1}(\B_\eta^{n-1} \backslash \{0\}) \cap \B_\e^m \to (\B_\eta^{n-1} \backslash \{0\}) \cap \Im(f_i^*)$$
and
$$(f_i)_|: (f_i)^{-1}(\B_\eta^1 \backslash \{0\}) \cap \B_\e^m \to (\B_\eta^1 \backslash \{0\}) \cap \Im(f_i)$$
are projections of smooth locally trivial fibrations. The converse $(ii) \implies (i)$ is an easy exercise, using Lemma \ref{lemma_1}.

\medskip
\item[$\bullet$] Let us show that $(i) \implies (iii)$. We will proceed by induction on $n$.

\medskip
\noindent Let $f: (\R^m,0) \to (\R^2,0)$ be as above and suppose that the restriction $f_|: f^{-1}(\B_\eta^2 \backslash \{0\}) \cap \B_\e^m \to (\B_\eta^2 \backslash \{0\}) \cap \Im(f)$ is the projection of a locally trivial fibration. Since $(i) \implies (ii)$, we have that for any $i=1,2$ and for any $t \in (\B_\eta^2 \backslash \{0\}) \cap \Im(f)$ with $t_i^* \neq 0$, the restrictions:
\medskip
\begin{itemize}
\item[$-$] $(f_i)_|: (f_i)^{-1}(\B_\eta^1 \backslash \{0\}) \cap \B_\e^m \to (\B_\eta^1 \backslash \{0\}) \cap \Im(f_i)$
\item[$-$] $(f_i)_| : (f_i)^{-1}(\B_\eta^1) \cap (f_i^*)^{-1}(t_i^*) \cap \B_\e^m \to \B_\eta^1 \cap \Im(f_i)$
\end{itemize}
\medskip
are projections of smooth locally trivial fibrations. In particular, for any $i=1, 2$ we have that $(f_i)^{-1}(t_i) \cap \B_\e^m$ is independent of $t_i \in (\B_\eta^1 \backslash \{0\}) \cap \Im(f_i)$. Moreover, for any $t= (t_1, t_1) \in (\B_\eta^2 \cap \Im(f))$ with $t_1 \neq 0$ we have that $(f_1)^{-1}(t_1) \cap \B_\e^m = (f_2^*)^{-1}(t_1) \cap \B_\e^m$ is diffeomorphic to $[f^{-1}(t) \cap \B_\e^m] \times \B_\eta^1$, which is diffeomorphic to $(f_1^*)^{-1}(t_1) \cap \B_\e^m = (f_2)^{-1}(t_1) \cap \B_\e^m$. 

\medskip
\noindent Now suppose that the result is true for $n-1$. Let $f: (\R^m,0) \to (\R^n,0)$ be such that the restriction $f_|: f^{-1}(\B_{\eta}^n \backslash \{0\}) \cap \B_\e^m \to (\B_{\eta}^n \backslash \{0\}) \cap \Im(f)$ is the projection of a smooth locally trivial fibration. Since $(i) \implies (ii)$, we have that the restriction $(f_i^*)_|: (f_i^*)^{-1}(\B_{\eta}^{n-1} \backslash \{0\}) \cap \B_\e^m \to (\B_{\eta}^{n-1} \backslash \{0\}) \cap \Im(f_i^*)$ is the projection of a locally trivial fibration, for any $i = 1, \dots, n$. 

\medskip
\noindent So we can apply the induction hypothesis to $f_j^*$, for some $j \in \{1, \dots, n\}$ fixed. We get that the diffeomorphism type of $(f_i)^{-1}(t_i) \cap \B_\e^m$ is independent of $i$, for any $i= \{1, \dots, n\} \setminus \{j\}$ and for any $t_i \in \B_{\eta}^1 \backslash \{0\}$, and also that it is diffeomorphic to $[(f_j^*)^{-1}(t_j^*) \cap \B_\e^m] \times \B_{\eta}^{n-2}$, for any $t_j^* \in \B_{\eta} \backslash \{0\}$. But by condition $(e)$ we have that $(f_j^*)^{-1}(t_j^*) \cap \B_\e^m$ is diffeomorphic to $[f^{-1}(t) \cap \B_\e^m] \times \B_{\eta}^1$. So $(f_i)^{-1}(t_i) \cap \B_\e^m$ is diffeomorphic to $[f^{-1}(t) \cap \B_\e^m] \times \B_{\eta}^{n-1}$.

\medskip
\noindent We can also apply the induction hypothesis to $f_k^*$, for some $k \in \{1, \dots, n\} \backslash \{j\}$. This gives that the diffeomorphism type of $(f_i)^{-1}(t_i) \cap \B_\e^m$ is independent of $i$, for any $i= \{1, \dots, n\} \backslash \{k\}$ and for any $t_i \in (\B_\eta^1 \backslash \{0\}) \cap \Im(f_i)$.

\medskip
\item[$\bullet$] The implication $(iii) \implies (iv)$ is obvious.

\medskip
\item[$\bullet$] Let us show that $(iv) \implies (ii)$. We will proceed by induction on $n$ again.

\medskip
First let us show that the result is true for $n=1$. That is, let $f: (\R^m,0) \to (\R,0)$ be a real analytic function-germ with isolated critical value such that the Euler characteristic of $f^{-1}(t) \cap \BS_\e^{m-1}$ is constant, for any $t \in (\B_\eta^1 \backslash \{0\}) \cap \Im(f)$. By Morse theory, this implies that the restriction:
$$f_|: f^{-1}(\B_\eta^1 \backslash \{0\}) \cap \BS_\e^{m-1} \to \B_\eta^1 \backslash \{0\}$$
is a submersion. Hence it follows from Ehresmann's fibration lemma that the restriction $f_|: f^{-1}(\B_\eta^1 \backslash \{0\}) \cap \B_\e^m \to \B_\eta^1 \backslash \{0\} \cap \Im(f)$ is the projection of a smooth locally trivial fibration.

\medskip
\noindent Now let $f: (\R^m,0) \to (\R^n,0)$ be a real analytic map-germ with isolated critical value and with $n>1$, such that the Euler characteristic of $f^{-1}(t) \cap \BS_\e^{m-1}$ is constant, for any $t \in (\B_{\eta}^n \backslash \{0\}) \cap \Im(f)$. Suppose that the result is true for $n-1$. 

\medskip
\noindent So for any $i=1, \dots, n$ fixed and for $t_i^* \neq 0$, it follows from Morse theory that the restriction:
$$(f_i)_|: (f_i)^{-1}(\B_{\eta}^1) \cap (f_i^*)^{-1}(t_i^*) \cap \BS_\e^{m-1} \to \B_{\eta}^1 \cap \Im(f_i)$$
is a submersion. But since the restriction:
$$(f_i)_|: (f_i)^{-1}(\B_{\eta}^1) \cap (f_i^*)^{-1}(t_i^*) \cap \mathring{\B}_\e^m \to \B_{\eta}^1 \cap \Im(f_i)$$
is also a submersion, it follows from the Ehresmann's fibration lemma that the restriction:
$$(f_i)_| : (f_i)^{-1}(\B_{\eta}^1) \cap (f_i^*)^{-1}(t_i^*) \cap \B_\e^m \to \B_{\eta}^1 $$
is a smooth locally trivial fibration over its image. In particular, this implies that for any $t_i^* \in (\B_\eta^{n-1} \backslash \{0\}) \cap \Im(f_i^*)$ the Euler characteristic of $(f_i^*)^{-1}(t_i^*) \cap \BS_\e^{m-1}$ equals the Euler characteristic of some $f^{-1}(t) \cap \BS_\e^{m-1}$, which is constant by hypothesis.

\medskip
\noindent Hence we can apply the induction hypothesis on the map $f_i^*$. So the restriction:
$$(f_i^*)_|: (f_i^*)^{-1}(\B_{\eta}^{n-1} \backslash \{0\}) \cap \B_\e^m \to (\B_{\eta}^{n-1} \backslash \{0\}) \cap \Im(f_i^*)$$
is the projection of a smooth locally trivial fibration. 

\medskip
\noindent Doing this procedure successively, we obtain that, for any $i=1, \dots, n$, the restriction:
$$(f_i)_|: (f_i)^{-1}(\B_{\eta}^1 \backslash \{0\}) \cap \B_\e^m \to (\B_{\eta}^1 \backslash \{0\}) \cap \Im(f_i)$$
is the projection of smooth locally trivial fibrations. This finishes the proof.
\end{itemize}

\medskip
Applying Theorem \ref{theo_main} successively, we obtain Corollary \ref{cor_2}.

\section{On the Euler characteristic with compact support of the link}

\medskip
Recall the Euler characteristic with compact support $\chi_C(X)$ of a triangulable space $X$, which is defined as the alternating sum:
$$\chi_C(X) := \sum_{k \geq 0} (-1)^k \dim_\R H_C^k(X,\R)  \, ,$$
where $H_C^\bullet (X,\R)$ denotes the cohomology with compact supports and real coefficients of the space $X$. 

Also recall that the Euler characteristic with compact support has the following properties (see \cite{ER} and \cite{GZ} for instance):

\begin{itemize}
\item[$\circ$] $\chi_C(\{point\}) =1$  ;
\item[$\circ$] $\chi_C(X) = \chi_C(Y)$ if $X$ is homeomorphic to $Y$;
\item[$\circ$] $\chi_C(X) = \chi_C(Y)$ for any homotopic compact spaces $X$ and $Y$;
\item[$\circ$] $\chi_C(X) = \chi_C(A) + \chi_C(X \backslash A)$ for every closed subset $A \subset X$;
\item[$\circ$] $\chi_C(X \times Y) = \chi_C(X) \chi_C(Y)$;
\item[$\circ$] If $M$ is an $n$-dimensional manifold (not necessarily compact), then $\chi_C(M) = (-1)^n \chi(M)$, where $\chi(M)$ is the Euler characteristic of $M$.
\end{itemize}
So, for any $n \in \mathbb{N}$, one has that:
$$
\left\{
\begin{array}{ccc}
\chi_C(\B^n) & = & 1 \\
\chi_C(\mathring{\B}^n) = \chi_C(\R^n) & = & (-1)^n \\
\chi_C(\BS^n) & = & 1 + (-1)^n \\
\end{array}
\right.
$$
 
\medskip
Now, let $f: (\R^m,0) \to (\R^n,0)$ be a real analytic map-germ, with $1<n<m$, and suppose that $f$ has isolated critical value and a L\^e-Milnor fibration in the tube. Let $F$ denote the Milnor fiber of $f$, whose interior is denoted by $\mathring{F}$ and whose boundary is denoted by $\partial F$.

Consider the real analytic function-germ:
$$
\begin{array}{cccc}
\| f \| \ : & \! (\R^m,0) & \! \longrightarrow & \! (\R,0) \\
& \! x & \! \longmapsto & \! \|f(x)\|
\end{array}
\, .$$
By Lemma \ref{lemma_0}, the function-germ $ \| f \|$ has a L\^e-Milnor fibration with Milnor fiber: 
$$H_s := (\|f\|)^{-1}(s) \cap \B_\e^m \, ,$$
where $s$ is a point in the semi-interval $(0,\eta]$. We set $\mathring{H}_s := (\|f\|)^{-1}(s) \cap \mathring{\B}_\e^m$, the interior of $H_s$, and we also set $K:=  (\|f\|)^{-1}(0) \cap \BS_\e^{m-1}$, the link of $\|f\|$, which coincides with the link of $f$.

Since $ (\|f\|)^{-1}\big( [0,\eta) \big) \cap \mathring{\B}_\e^m$ is homeomorphic to an open ball in $\R^m$, it follows that:
$$\chi_C \big(  (\|f\|)^{-1}\big( [0,\eta) \big) \cap \mathring{\B}_\e^m \big) = (-1)^m \, .$$

\medskip
So if $\mathring{H}_0 := (\|f\|)^{-1}(0) \cap \mathring{\B}_\e^m$ we have that:

\begin{equation} \label{eq1}
\chi_C \big(  (\|f\|)^{-1}\big( (0,\eta) \big) \cap \mathring{\B}_\e^m \big) + \chi_C(\mathring{H}_0) = (-1)^m \, .
\end{equation}

\medskip
On the one hand, we know that $H_0$ is homeomorphic to the cone over $K$, and hence: 
$$\chi_C(\mathring{H}_0) = -\chi_C(K) +1 \, .$$
On the other hand, since the restriction of $ \| f \|$ to $ (\|f\|)^{-1}\big( (0,\eta) \big) \cap \mathring{\B}_\e^m$ is a locally trivial fibration over $(0, \eta)$, we have that: 
$$\chi_C \big(  (\|f\|)^{-1}\big( (0,\eta) \big) \cap \mathring{\B}_\e^m \big) = - \chi_C(\mathring{H}_s) \, .$$
Putting this on equation (\ref{eq1}) we obtain:

\begin{equation} \label{eq2}
\chi_C(K) = -\chi_C(\mathring{H}_s) + 1 + (-1)^{m+1}\, .
\end{equation}

\medskip
We also know that $\mathring{H}_s$ is a fiber bundle over $\BS^{n-1}$ with fiber $\mathring{F}$. Hence we have:

\begin{equation} \label{eq3}
\chi_C(\mathring{H}_s) = \big( 1+(-1)^{n-1} \big) \chi_C(\mathring{F}) \, .
\end{equation}

\medskip
So equations (\ref{eq2}) and (\ref{eq3}) together give:
$$\chi_C(K) = -\big( 1+(-1)^{n-1} \big) \chi_C(\mathring{F}) + \big( 1+(-1)^{m+1} \big)   \, .$$
That is:

\begin{equation} \label{eq4}
\chi_C(K) =
\begin{cases}
\hspace{1.2cm} 0 \hspace{1.2cm} , \ {\text if} \ m \in 2 \Z \ \ {\text and} \ n \in 2 \Z \ ; \\
\hspace{.42cm} -2 \chi_C(\mathring{F}) \hspace{.42cm} , \ {\text if} \ m \in 2 \Z \ \ {\text and} \ n \in 2 \Z+1 \ ; \\
\hspace{1.21cm} 2 \hspace{1.21cm} , \ {\text if} \ m \in 2 \Z+1 \ \ {\text and} \ n \in 2 \Z \ ; \\
-2 \chi_C(\mathring{F})  +2 \ , \ {\text if} \ m \in 2 \Z+1 \ \ {\text and} \ n \in 2 \Z+1 \ . \\
\end{cases} 
\end{equation}

\medskip
Since $\chi_C(\mathring{F}) = (-1)^{m-n} \chi(\mathring{F}) = (-1)^{m-n} \chi(F)$, this proves Theorem \ref{theo_link} for $n>1$. For $n=1$, one just has to substitute equation \ref{eq3} by the equation:

\begin{equation} \label{eq3}
\chi_C(\mathring{H}_s) = \chi_C(\mathring{F}_+) + \chi_C(\mathring{F}_-) \, .
\end{equation}

This finishes the proof of Theorem \ref{theo_link}.

\medskip
Now let us restrict to the case when $f$ has an isolated critical point at $0 \in \R^m$.

\medskip
Suppose that $n>1$ and that $m$ is odd. If we further suppose that $n$ is odd, we have that $\chi(K) =0$, so it follows from Theorem \ref{theo_link} that $\chi(F)  = 1$. This, together with Theorem \ref{theo_main}, proves $(i)$ of Corollary \ref{cor_EC}.

\medskip
Now suppose that $n=1$ and that $m$ is odd. Then again we have that $\chi(K) =0$, so Theorem \ref{theo_link} gives that $\chi(F_+) + \chi(F_-) = 2$, which is $(ii)$ of Corollary \ref{cor_EC}.

\medskip
Finally, suppose that $n=1$ and that $m$ is even. A well-known topological argument gives that: 
$$2 \chi(F_+) = \chi(\partial F_+) = \chi(K) = \chi(\partial F_-) = 2 \chi(F_-) \, .$$
Hence $\chi(F_+) = \chi(F_-)$, which is $(iii)$ of Corollary \ref{cor_EC}.

\medskip
\section{Answering a question}

\medskip
Let us keep the notation of the previous sections. 

In \cite{CSS2}, Cisneros, Seade and Snoussi introduced the concept of $d$-regularity for a real analytic map-germ $f: (\R^m,0) \to (\R^n,0)$, with $1<n \leq m$ and with an isolated critical value at $0 \in \R^n$. They associated a canonic pencil to $f$, with axis $V:= f^{-1}(0)$, such that the elements of this pencil are all analytic varieties with singular set contained in $V$. The map-germ $f$ is $d$-regular if away from the axis each element of the pencil is transverse to all sufficiently small spheres.

For each $i = 1, \dots, n$ set $V_i^* := (f_i^*)^{-1}(0)$. In \cite{SCG} Cisneros-Molina, Seade and Grulha Jr. proved the following:

\begin{theo} \label{theo_SCG}
Let $f: (\R^m,0) \to (\R^n,0)$, with $1<n<m$, be a complete intersection map-germ with an isolated critical point at $0 \in \R^m$. If $f$ is $d$-regular, then the topology of $V_i^* \backslash V$ is independent of the choice of $i$ and its link, which is a smooth manifold, is diffeomorphic to the disjoint union of two copies of the interior of the Milnor fiber of $f$.
\end{theo}

This motivated them to make the following question: Is the topology of $V_i^*$ independent of $i$ as in the complex case? They pointed that if this is true, then the link of $V_i^*$ is diffeomorphic to the double of the Milnor fiber of $f$.

\medskip
Using Corollary \ref{cor_2} we can easily answer that question. We have:

\begin{theo}
Let $f = (f_1, \dots ,f_n) : (\R^m,0) \to (\R^n,0)$, with $1<n<m$, be a map-germ with an isolated critical point. For each set of indices $\{ i_1, \dots, i_k \}$ set: 
$$V_{i_1, \dots, i_k} := (f_{i_1, \dots, i_k})^{-1}(0) \cap \B_\e^m \, $$
where $f_{i_1, \dots, i_k} = (f_{i_1}, \dots, f_{i_k}): (\R^m,0) \to (\R^k,0)$ and $\e$ is a Milnor ball for $f$. Then the topology of $V_{i_1, \dots, i_k}$ is independent of the choice of the indices $i_1, \dots, i_k$. Moreover, the diffeomorphism type of its link is also independent of the choice of the indices $i_1, \dots, i_k$, and it is diffeomorphic to the boundary of the product of the Milnor fiber $F$ of $f$ and a closed $(n-k)$-disk. In the particular case when $k=n-1$, that is precisely the double of $F$.
\end{theo}

\begin{proof}
Since $f$ has an isolated critical point, it has a L\^e-Milnor fibration:
$$f_|: f^{-1}(\B_\eta^n \backslash \{0\}) \cap \B_\e^m \to \B_\eta^n \backslash \{0\} \, .$$

So Corollary \ref{cor_2} gives that the Milnor fiber $F_{i_1, \dots, i_k}$ of $f_{i_1, \dots, i_k}$ is diffeomorphic to the product $F \times \B_\eta^{n-k}$. In particular, its diffeomorphism type is independent of the choice of the indices $i_1, \dots, i_k$.

But $\Sigma(f_{i_1, \dots, i_k}) \subset \Sigma(f)$ and then $f_{i_1, \dots, i_k}$ has an isolated critical point. Hence its link $K_{i_1, \dots, i_k} := V_{i_1, \dots, i_k} \cap \BS_\e^{m-1}$ is diffeomorphic to the boundary of $F_{i_1, \dots, i_k}$, which is diffeomorphic to the boundary of $F \times \B_\eta^{n-k}$. So the diffeomorphism type of $K_{i_1, \dots, i_k}$ is independent of the choice of the indices $i_1, \dots, i_k$.

Since $V_{i_1, \dots, i_k}$ is homeomorphic to the cone over $K_{i_1, \dots, i_k}$ we also have that the homeomorphism type of $V_{i_1, \dots, i_k}$ is independent of the choice of $i_1, \dots, i_k$.

\end{proof}


\end{document}